\documentclass[a4paper,10pt]{article}
\usepackage{amsmath, amsthm, amssymb, amsfonts, mathrsfs, amscd, color, graphicx}
\usepackage[numbers,sort]{natbib}

\definecolor{extlinkz}{rgb}{0,0,1}
\definecolor{intlinkz}{rgb}{1,0,0}
\definecolor{citlinkz}{rgb}{0,.6,.1}

\usepackage[pdfpagelabels,plainpages=false,colorlinks=true,citecolor=citlinkz,linkcolor=intlinkz,urlcolor=extlinkz,urlbordercolor={0 0 0}]{hyperref}

\newtheorem{theo}{Theorem}
\newtheorem{defi}[theo]{Definition}
\newtheorem{lem}[theo]{Lemma}
\newtheorem{prop}[theo]{Proposition}
\newtheorem{cor}[theo]{Corollary}
\newtheorem{rem}[theo]{Remark}

\newcommand{\wh}[1]{\widehat{#1}}

\newcommand{\ol}[1]{\overline{#1}}

\newcommand{\R}{\mathbb{R}}
\newcommand{\Z}{\mathbb{Z}}
\newcommand{\T}{\mathbb{T}}
\newcommand{\C}{\mathbb{C}}
\newcommand{\N}{\mathbb{N}}
\newcommand{\F}{\mathcal{F}}

\newcommand{\Hil}{\mathcal{H}}

\newcommand{\ind}{\mathcal{I}}

\newcommand{\psis}[1]{\langle #1\rangle_\Gamma}

\newcommand{\vsp}{\textnormal{span}}

\newcommand{\Ran}{\textnormal{Ran}}
\newcommand{\Ker}{\textnormal{Ker}}

\newcommand{\rr}{\rho}
\newcommand{\vn}{\mathscr{R}}
\newcommand{\lr}{\lambda}
\newcommand{\vnL}{\mathscr{L}}

\newcommand{\id}{{\mathrm{e}}}
\newcommand{\orb}{{\mathcal{O}_{\Gamma}(\psi)}}
\newcommand{\g}{{g_{{}_\orb}}}

\newcommand{\Sy}{\textnormal{T}}

\newcommand{\An}{\Sy^*}

\newcommand{\Grop}{\mathfrak{G}}

\newcommand{\Gram}{\mathcal{G}}
\newcommand{\Frame}{\mathfrak{F}}

\newcommand{\Id}{\mathbb{I}}
\newcommand{\Proj}{\mathbb{P}}

\definecolor{highlight}{rgb}{1,0,.6}
\definecolor{suggest}{rgb}{0,.7,.9}

\title{Group Riesz and Frame Sequences:\\ The Bracket and the Gramian.}
\author{Davide Barbieri, Eugenio Hern\'andez, Victoria Paternostro}

\begin{document}

\maketitle

\vspace{-6pt}
\begin{abstract}
Given a discrete group and a unitary representation on a Hilbert space $\Hil$, we prove that the notions of operator Bracket map and Gramian coincide on a dense set of $\Hil$. As a consequence, combining this result with known frame theory, we can recover in a simple way all previous Bracket characterizations of Riesz and frame sequences generated by a single element under a unitary representation.
\end{abstract}

\section{Introduction} \label{Intro}

Riesz and frame sequences in a separable Hilbert space $\Hil$ are key objects in approximation theory. The following definitions are standard and can be found e.g. in \cite{Daubechies92, Meyer93, HW96, Christensen03}. For a countable set of indices $\ind$ consider the family  $\Psi = \{\psi_j\}_{j \in \ind} \subset \Hil$ and call $\Hil_\Psi = \ol{\vsp(\Psi)}^\Hil$.  The system $\Psi$ is said to be a \textbf{Riesz sequence} with Riesz bounds $0 < A \leq B < \infty$ if it satisfies
\begin{equation}\label{eq:Riesz}
A \|c\|^2_{\ell_2(\ind)} \leq \|\sum_{j \in \ind} c_j\psi_j\|^2_{\Hil} \leq B \|c\|^2_{\ell_2(\ind)}
\end{equation}
for all finite sequence $c = \{c_j\}_{j \in \ind} \in \ell_0(\ind)$. Since finite sequences are dense in $\ell_2(\ind)$, this condition is equivalent to say that (\ref{eq:Riesz}) holds for all $c \in \ell_2(\ind)$.  Recall that a Riesz sequence is a \emph{Riesz basis for} $\Hil_\Psi$.

The system $\Psi$ is said to be a \textbf{frame sequence} with frame bounds $0 < A \leq B < \infty$ if it satisfies
\begin{equation}\label{eq:frames}
A \|\varphi\|^2_{\Hil} \leq \sum_{j \in \ind} |\langle \varphi, \psi_j \rangle_{\Hil} |^2 \leq B \|\varphi\|^2_{\Hil}, \quad \mbox{for all}\ \varphi \in \Hil_\Psi.
\end{equation}
 Recall that a frame sequence is a \emph{frame for} $\Hil_\Psi$.

\

Particular systems, that play an important role in the construction of wavelets from a Multiresolution Analysis, in the study of Gabor systems, and in approximation theory, are formed by translations of a single function $\psi\in L^2(\R^d).$ For $k\in \Z^d$ denote by $T_k\psi (x) =\psi(x-k)$ the translation of $\psi$ by the integer $d$-tuple $k$.

When does $\{T_k \psi: k \in \Z^n\}$ constitute a Riesz sequence in $L^2(\R^d)$? When does $\{T_k \psi: k \in \Z^n\}$  constitute a frame sequence in $L^2(\R^d)$?

A crucial object to answer this questions is the Bracket map, that is the sesquilinear map from $L^2(\R^n) \times L^2(\R^n)$ to $L^1([0,1])$ that reads
\begin{equation}\label{eq:originalBracket}
[\varphi,\psi](\xi) = \sum_{j \in \Z^n} \wh{\varphi}(\xi + j) \ol{\wh{\psi}(\xi + j)} ,
\end{equation}
where $\wh{\varphi}$ denotes the Fourier transform of $\varphi$.
In terms of this Bracket map the answers to this questions is the following:
\begin{prop}\label{prop:1}
	Let $\psi \in L^2(\R^d)$ and fix constants $0 < A \leq B < +\infty$. 
	\begin{itemize}
		\item[i.] The collection $\{T_k\psi : k \in \Z \}$ is a Riesz sequence with bounds $0 < A \leq B < \infty$ if and only if $A \leq [\psi , \psi](\xi) \leq B$ a.e. $\xi \in [0,1].$
		\item[ii.] The collection $\{T_k\psi : k \in \Z\}$ is a frame sequence with bounds $0 < A \leq B < \infty$ if and only if $A \leq [\psi , \psi](\xi) \leq B$ a.e. $\xi \in \mbox{supp}\, [\psi, \psi].$
	\end{itemize}
\end{prop}
The proof of part $i.$ of Proposition \ref*{prop:1} can be found in classical books on wavelets such as \cite{Meyer93, Daubechies92} and in \cite{Bow00}. As for the proof of part $ii.$ it can be found in \cite{BenedettoLi93,BenedettoWalnut94,BenedettoLi98, Bow00}.

\

Dilations by dyadic integers, modulations and shear transformations (see \cite{Zak-SIS} for definitions) are other examples for which a similar result to the one contained in Proposition \ref{prop:1} can be given. All these are examples of unitary group representations. A unitary representation of a group $\Gamma$ is a continuous group homomorphism $\Pi: \Gamma \longrightarrow U(\Hil)$, where $U(\Hil)$ denotes the set of unitary operators of the Hilbert space $\Hil$.

Given a unitary representation $\Pi$ of a discrete group $\Gamma$, and a nonzero $\psi \in \Hil$, two natural questions are then the following. When does $\{\Pi(\gamma) \psi: \gamma \in \Gamma\}$ constitute a Riesz sequence in $\Hil$? When does $\{\Pi(\gamma) \psi: \gamma \in \Gamma\}$ constitute a frame sequence in $\Hil$? The answer to these questions is given for the abelian case in \cite{HSWW10} and for the nonabelian case in \cite{BHP14}. The answers are formally different, but it is expressed in both cases in terms of an important object that generalizes the Bracket map (\ref{eq:originalBracket}). We review now the abelian case.

Suppose that $\Gamma$ is an abelian discrete and countable group. Let $\wh{\Gamma}$ be the dual group of $\Gamma$, that is, the group of all characters of $\Gamma$ defined as the continuous maps $\alpha: \Gamma \longrightarrow \C$ for which $|\alpha(\gamma)|=1$ for all $\gamma\in \Gamma$ and $\alpha(\gamma_1 \gamma_2)= \alpha(\gamma_1)\alpha(\gamma_2)$ for all $\gamma_1, \gamma_2 \in \Gamma$. Since $\Gamma$ is discrete, $\wh \Gamma$ is compact. Denote by $d\alpha$ the normalized Haar measure on $\wh{\Gamma}$. A unitary representation $\Pi: \Gamma \longrightarrow U(\Hil)$ is said to be \textbf{dual integrable} if there exists a sesquilinear map
$$
[\cdot , \cdot] : \Hil \times \Hil \rightarrow L^1(\wh{\Gamma}, d\alpha)
$$
such that
\begin{equation} \label{eq:Bracketabelian}
\big\langle \varphi, \Pi(\gamma) \psi \big\rangle_{\Hil} = \int_{\widehat{\Gamma}} [\varphi, \psi](\alpha)\alpha(\gamma) d\alpha \quad \forall \ \varphi, \psi \in \Hil \quad \forall \ \gamma \in \Gamma.
\end{equation}

\noindent Let us illustrate this definition with two basic examples. In the case of integer translations, the Bracket is given by equation (\ref{eq:originalBracket}) (see \cite{BoorVoreRon93, BoorVoreRon94, BenedettoLi98}).
For the case of the Gabor representation of the group $(\Z^d \times \Z^d,+)$ on $L^2(\R^d)$ given by $M_l T_k \psi(x) = e^{-2\pi i l\cdot x} \psi(x - k)$, the Bracket is $[\varphi, \psi](x,\xi) = Z\varphi(x,\xi)\overline{Z\psi(x,\xi)}$, where $x,\xi \in \T^d$, $\varphi, \psi \in L^2(\R^d)$ and
	\begin{equation} \label{eq:Zakclassical}
	Z\psi(x,\xi) = \sum_{k \in \Z^d}  \psi(k + x)\, e^{-2\pi i k \cdot \xi}
	\end{equation}
	is the Zak transform of the function $\psi \in L^2(\R^d)$ (see for instance \cite{HeilPowell06}).

For a dual integrable representation the answer to the above questions is similar to the one contained in Proposition \ref{prop:1}.

\begin{prop} (\cite{HSWW10})\label{prop:conmutative}
Let $\Pi$ be a dual integrable representation of a discrete and countable  abelian group $\Gamma$ on a Hilbert space $\Hil$.	Let $\psi \in \Hil$ and fix constants $0 < A \leq B < +\infty$. 
\begin{itemize}
\item[i.] The collection is a Riesz sequence in $\Hil$ with bounds $0 < A \leq B < \infty$ if and only if $A \leq [\psi , \psi](\alpha) \leq B$ a.e. $\alpha \in \wh \Gamma.$
\item[ii.] The collection $\{\Pi(\gamma)\psi : \gamma \in \Gamma \}$  is a frame sequence in $\Hil$ with bounds $0 < A \leq B < \infty$ if and only if $A \leq [\psi , \psi](\alpha) \leq B$ a.e. $\alpha \in \textnormal{supp}\, [\psi, \psi].$
\end{itemize}
\end{prop}

\

The nonabelian case requires more machinery that was introduced in (\cite{BHP14}) and will be reviewed in Section \ref{Section2}. In this paper we will show that this Bracket notion for the nonabelian case coincides with the well-known Gramian, that can be defined for any countable family $\{\psi_j\}_{j \in \ind}$ satisfying some weak square integrability conditions as the closed and densely defined operator on $\ell_2(\ind)$ whose matrix coefficients are the inner products $\langle \psi_j, \psi_i\rangle$.

In order to do so, we will start recalling in Section  \ref{Section2} the notion of operator Bracket map  introduced in \cite{BHP14} in terms of the group von Neumann algebra. In Section \ref{sec:basicresults} the Gramian is defined for a countable family under the weakest possible conditions. In Section \ref{GramianFrames} we provide a simple argument for the characterization of general frame sequences  (see Corollary \ref{cor:framegramian}), as a consequence of Proposition \ref{lem:duallemma}. The main result of the paper is contained in Section \ref{Braket=Gramian}, where we show that the operator Bracket map introduced in \cite{BHP14} for any discrete groups coincides with the Gramian. In Section \ref{abelian} we then show that, for abelian groups, operator Bracket map and the Bracket map (\ref{eq:Bracketabelian}) are related by Pontryagin duality. As corollaries we obtain all known characterizations of Riesz and frame sequences for principal invariant subspaces; this gives simple proofs of these results, hence avoiding the - as the authors term them - ``surprisingly intricate'' arguments invoked in \cite{BenedettoLi98} and still present in all subsequent works.

\section{Preliminaries} \label{Section2}

\subsection{The nonconmutative setting and the Bracket map} \label{Subsection2-1}

In this section, we will describe the setting where we will work in and recall the notion of our main object of study: the operator Bracket map, introduced in \cite{BHP14}.

Given $\Gamma$  a discrete and countable group 
 we will need to deal with Fourier analysis over the von Neumann algebra associated to $\Gamma$ which we briefly recall here (for details see  \cite{KadisonRingrose83, Conway00, Connes94, Takesaki03, PisierXu03, JumgeMeiParcet} and the discussion in \cite{BHP14}). In this case, we  shall consider 
the \emph{right} von Neumann algebra of $\Gamma$, defined as follows. Let $\rr : \Gamma \to U(\ell_2(\Gamma))$ be the right regular representation, which acts on the canonical basis $\{\delta_\gamma\}_{\gamma \in \Gamma}$ as $\rr(\gamma)\delta_{\gamma'} = \delta_{\gamma'\gamma^{-1}}$, and let us call \emph{trigonometric
	polynomials} the operators obtained by finite linear combinations of $\{\rr(\gamma)\}_{\gamma \in \Gamma}$. The right von Neumann algebra can be defined as the weak operator closure of such trigonometric polynomials
$$
\vn(\Gamma) = \ol{\vsp\{\rr(\gamma)\}_{\gamma \in \Gamma}}^{\textrm{WOT}}.
$$
Given $F \in \vn(\Gamma)$, we will denote by $\tau$ the standard trace
$$
\tau(F) = \langle F \delta_\id, \delta_\id \rangle_{\ell_2(\Gamma)}
$$
where $\id$ denotes the identity element of $\Gamma$.  Then, $\tau$ defines a normal, finite and faithful tracial linear functional.

For  $F \in \vn(\Gamma)$, the Fourier coefficients of $F$, $\{\wh{F}(\gamma)\}_{\gamma\in\Gamma} \in \ell_2(\Gamma)$, are given by
\begin{equation}\label{eq:Fouriercoefficients}
\wh{F}(\gamma) = \tau(F \rr(\gamma)), 
\end{equation}
and $F$ has a \emph{Fourier series} 
$
F = \sum_{\gamma \in \Gamma} \wh{F}(\gamma) \rr(\gamma)^* \, ,
$
which converges in the weak operator topology.

Any $F \in \vn(\Gamma)$ is a bounded right convolution operator by $\wh{F}$: given $u \in \ell_2(\Gamma)$
$$
Fu(\gamma) = \sum_{\gamma' \in \Gamma} \wh{F}(\gamma')\rr(\gamma')^*u(\gamma) = \sum_{\gamma' \in \Gamma} \wh{F}(\gamma')u(\gamma\gamma'^{-1}) = u \ast \wh{F}(\gamma)\vspace{-4pt}
$$
where $\ast$ stands for  the $\Gamma$-group convolution
$$
u \ast v (\gamma) = \sum_{\gamma' \in \Gamma} u(\gamma\gamma'^{-1})v(\gamma') = \sum_{\gamma' \in \Gamma} u(\gamma') v(\gamma'^{-1}\gamma) .\vspace{-4pt}
$$

For any $1 \leq p < \infty$ let $\|\cdot\|_p$ be the norm over $\vn(\Gamma)$ given by\vspace{-4pt}
$$
\|F\|_p = \tau(|F|^p)^\frac1p\vspace{-4pt}
$$
where the absolute value is the selfadjoint operator defined as $|F| = \sqrt{F^*F}$ and the $p$-th power is defined by functional calculus of $|F|$. Following \cite{Nelson74, PisierXu03, BHP14}, we define the noncommutative $L^p(\vn(\Gamma))$ spaces for $1 \leq p < \infty$ as
$$
L^p(\vn(\Gamma)) = \ol{\vsp\{\rr(\gamma)\}_{\gamma \in \Gamma}}^{\|\cdot\|_p}
$$
while for $p = \infty$ we set $L^\infty(\vn(\Gamma)) = \vn(\Gamma)$ endowed with the operator norm.

When $p < \infty$, the elements of $L^p(\vn(\Gamma))$ are the linear operators on $\ell_2(\Gamma)$ that are affiliated to $\vn(\Gamma)$, i.e. the densely defined closed operators that commute with all unitary elements of $\vnL(\Gamma)$, whose $\|\cdot\|_p$ norm is finite (see also \cite{Terp81}). Here, $\vnL(\Gamma)$ denotes the \emph{left} von Neumann algebra of $\ell_2(\Gamma)$, that is generated by the left regular representation $\lr: \Gamma \to U(\ell_2(\Gamma))$, defined by $\lr(\gamma)\delta_{\gamma'} = \delta_{\gamma\gamma'}$. 

In particular, $L^p(\vn(\Gamma))$ elements for $p < \infty$ are not necessarily bounded, while a bounded operator that is affiliated to $\vn(\Gamma)$ automatically belongs to $\vn(\Gamma)$ as a consequence of von Neumann's Double Commutant Theorem (see also \cite[Th. 4.1.7]{KadisonRingrose83}). For $p = 2$ one obtains a separable Hilbert space with scalar product\vspace{-3pt}
$$
\langle F_1, F_2\rangle_2 = \tau(F_2^* F_1)\vspace{-3pt}
$$
for which the monomials $\{\rr(\gamma)\}_{\gamma \in \Gamma}$ form an orthonormal basis.
For these spaces the usual statement of H\"older inequality still holds, so that in particular for any $F \in L^p(\vn(\Gamma))$ with $1 \leq p \leq \infty$ its Fourier coefficients are well defined, and the finiteness of the trace implies that $L^p(\vn(\Gamma)) \subset L^q(\vn(\Gamma))$ whenever $q < p$. Moreover, fundamental results of Fourier analysis such as $L^1(\vn(\Gamma))$ Uniqueness Theorem, Plancherel Theorem between $L^2(\vn(\Gamma))$ and $\ell_2(\Gamma)$, and Hausdorff-Young inequality still hold in the present setting (see e.g. \cite[\S 2.2]{BHP14}).

A relevant class of operators in $\vn(\Gamma)$ are the orthogonal projections onto closed subspaces $W$ of $\ell_2(\Gamma)$ 
 satisfying $\lr(\Gamma)W \subset W$, for all $\gamma\in\Gamma$. 
 A special case is the spectral projection over the set $\R \setminus \{0\}$, that is called the support of $F$. It is the minimal orthogonal projection $s_F$ of $\ell_2(\Gamma)$ such that $F = F s_F = s_F F$, and reads explicitly
\begin{equation}\label{eq:support}
s_F = \Proj_{(\Ker(F))^\bot} = \Proj_{\ol{\Ran(F)}} .
\end{equation}

Let us now recall the following  definition from \cite{HSWW10, BHP14} which is essential in this paper. 
\begin{defi}\label{noncommutativeBracket}
	Let $\Pi$ be a unitary representation of a discrete and coutable group $\Gamma$ on a separable Hilbert space $\Hil$. We say that $\Pi$ is \emph{dual integrable} if there exists a sesquilinear map $[\cdot,\cdot] : \Hil \times \Hil \to L^1(\vn(\Gamma))$, called operator \emph{Bracket map}, satisfying
	$$
	\langle \varphi, \Pi(\gamma)\psi\rangle_\Hil = \tau([\varphi,\psi]\rr(\gamma)) \quad \forall \, \varphi, \psi \in \Hil \, , \ \forall \, \gamma \in \Gamma .
	$$
	In such a case we will call $(\Gamma,\Pi,\Hil)$ a \emph{dual integrable triple}.
\end{defi}
Note that $[\varphi,\psi]$ is the object in $ L^1(\vn(\Gamma))$ which Fourier coefficites are $\{\langle \varphi, \Pi(\gamma)\psi\rangle_\Hil\big\}_{\gamma \in \Gamma}$.
 
According to \cite[Th. 4.1]{BHP14}, $\Pi$ is dual integrable if and only if it is square integrable, in the sense that there exists a dense subspace $\mathcal{D}$ of $\Hil$ such that
$$
\big\{\langle \varphi, \Pi(\gamma)\psi\rangle_\Hil\big\}_{\gamma \in \Gamma} \in \ell_2(\Gamma) \quad \forall \, \varphi \in \Hil \, , \ \forall \, \psi \in \mathcal{D}.
$$

\subsection{Gramian and frame operators}\label{sec:basicresults}

In this subsection we will provide some fundamental definitions and results concerning the key operators involved in the study of Riesz bases and frames.  For $\Psi = \{\psi_j\}_{j \in \ind} \subset \Hil$, the \emph{synthesis operator} of $\Psi$ is the densely defined operator from $\ell_2(\ind)$ to $\Hil$ which, on finite sequences, reads
$$
\begin{array}{rccc}
\Sy_\Psi : & \ell_0(\ind) & \rightarrow & \vsp(\Psi) \subset \Hil\vspace{4pt}\\
& c & \mapsto & \displaystyle\sum_{j \in \ind} c_j \psi_j .
\end{array}
$$

Its adjoint operator $\An\Psi$ (see e.g. \cite[Chapter X, \S 1]{Conway90}) is called the \emph{analysis operator} of $\Psi.$ It is the closed operator defined on $Dom(\An\Psi)=\{\varphi \in \Hil:\,\, c \mapsto \langle T_\psi c, \varphi \rangle_\Hil \textrm{ is a bounded linear functional on }Dom(\Sy_\Psi)= \ell_0(\ind)\}$, as
$$
\langle c, \An_\Psi \varphi \rangle_{\ell_2(\ind)} = \langle \Sy_\Psi c , \varphi \rangle_\Hil= \langle \sum_{j \in \ind} c_j \psi_j, \varphi\rangle_\Hil = \sum_{j \in \ind} c_j \langle \psi_j, \varphi\rangle_\Hil
$$	
$c\in \ell_0(\ind)$ and reads explicitly
$$
\begin{array}{rccc}
\An_\Psi : & Dom(\An_\Psi) & \rightarrow & \ell_{2}(\ind)\vspace{4pt}\\
& \varphi & \mapsto & \big\{\langle\varphi,\psi_j\rangle_\Hil\big\}_{j \in \ind}\,.
\end{array}
$$

\begin{lem}\label{lem:generalsquareintegrability}
Let $\Psi = \{\psi_j\}_{j \in \ind} \subset \Hil$ be a countable family. The analysis operator $\An_\Psi$ maps $\vsp(\Psi)$ to $\ell_2(\ind)$ if and only if 
the family $\Psi = \{\psi_j\}_{j \in \ind}$  satisfies the following square integrability condition: \begin{equation}\label{eq:generalsquareintegrability}
t_j = \sum_{k \in \ind} |\langle \psi_j, \psi_k\rangle_\Hil|^2 \ \textnormal{is finite for all} \ j \in \ind .
\end{equation}

\end{lem}
\begin{proof}
Let us first assume (\ref{eq:generalsquareintegrability}). We have to show the inclusion  $ \vsp(\Psi) \subset Dom(\An_\Psi)$. Since $Dom(\An_\Psi)$ is a linear subspace of $\Hil$ it is enough to show that $\psi_k$ belongs to $Dom(\An_\Psi)$ for all $k \in \ind.$ For $c\in \ell_0(\ind)$ the computation
\begin{align*}
|\langle \Sy_\Psi c , \psi_k \rangle_\Hil| & = |\langle \sum_{j\in \ind} c_j \psi_j , \psi_k \rangle_\Hil | \leq \sum_{j\in \ind} |c_j| |\langle \psi_j , \psi_k \rangle_\Hil | \\
& \leq (\sum_{j\in \ind} |c_j|^2)^{1/2}\, (\sum_{j\in \ind} |\langle \psi_j , \psi_k \rangle_\Hil|^2)^{1/2} = \|c\|_{\ell_2(\ind)} t_k^\frac12\, ,
\end{align*}
shows that the map $c \mapsto \langle \An_\Psi c , \psi_k \rangle_\Hil$ is a bounded linear functional on $Dom(\Sy_\Psi)=\ell_0(\ind)$. Hence $\psi_k \in Dom(\An_\Psi)$.
	
Conversely, since $\An_\Psi : \vsp(\Psi) \to \ell_2(\ind)$, then in particular
\begin{displaymath}
\|\An_\Psi \psi_j\|_{\ell_2(\ind)}^2 = t_j < \infty \quad \forall \, j \in \ind. \qedhere
\end{displaymath}
\end{proof}

Assuming (\ref{eq:generalsquareintegrability}), by composition of analysis and synthesis operators one obtains the \emph{Gramian} associated to $\Psi$ as a densely defined operator on $\ell_2(\ind)$ that on $\ell_0(\ind)$ reads
$$
\begin{array}{rccc}
\Grop_\Psi = \An_\Psi \Sy_\Psi : & \ell_0(\ind) & \rightarrow & \ell_2(\ind)\vspace{4pt}\\
& c & \mapsto & \displaystyle\Big\{\langle \sum_{j \in \ind} c_j \psi_j,\psi_k\rangle_\Hil\Big\}_{k \in \ind} .
\end{array}
$$
Its name is motivated by the observation that
$$
\Big(\Grop_\Psi c\Big)_k = \sum_{j \in \ind} c_j \langle \psi_j,\psi_k\rangle_\Hil  = \sum_{j \in \ind} \Gram_\Psi^{k,j} c_j \ , \quad k \in \ind
$$
where $\Gram_\Psi = (\Gram_\Psi^{k,j}) = (\langle \psi_j, \psi_k\rangle_\Hil)$ is the Gram matrix of $\Psi$ (note the ordering of indices).

\begin{cor}\label{cor:closableGram}
	Let $\Psi = \{\psi_j\}_{j \in \ind} \subset \Hil$ be a countable family for which  (\ref{eq:generalsquareintegrability}) holds. Then, the Gramian is a closable densely defined operator on $\ell_2(\ind)$ whose domain contains finite sequences, i.e.
	$
	\Grop_\Psi : \ell_0(\ind) \rightarrow \ell_2(\ind) .
	$
\end{cor}
\begin{proof}
	$\Grop_\Psi$ is densely defined in $\ell_2(\ind)$ by Lemma \ref{lem:generalsquareintegrability}. Moreover, since $\An_\Psi$ is densely defined on $\Hil_\Psi$, then $\Sy_\Psi$ is closable, and its closure is given by $\Sy_\Psi^{**}$ (see e.g. \cite[Chapter X, \S 1]{Conway90}). To see that also $\Grop_\Psi$ is closable, let $\{f^n\}_{n \in \N} \subset \ell_0(\ind)$ be a sequence converging to $f \in \ell_2(\ind)$ such that $\{\Grop_\Psi f^n\}_{n \in \N}$ converges to $g \in \ell_2(\ind)$. This implies that $\{\Sy_\Psi f^n\}_{n \in \N} \subset \Hil$ is convergent, because
	\begin{align*}
	\|\Sy_\Psi f^n - \Sy_\Psi f^m\|^2_{\Hil} & = \langle f^n - f^m, \Grop_\Psi(f^n - f^m)\rangle_{\ell_2(\ind)}\\
	& \leq \|f^n - f^m\|_{\ell_2(\ind)} \|\Grop_\Psi(f^n - f^m)\|_{\ell_2(\ind)} .
	\end{align*}
	Since $\Sy_\Psi$ is closable, then there exists $\varphi \in \Hil$ such that $\Sy_\Psi^{**} f = \varphi$, while the closedness of $\An_\Psi$ implies that $\An_\Psi \varphi = g$, so the extension of the Gramian defined by $\An_\Psi\Sy_\Psi^{**}$ is closed.
\end{proof}

\noindent
Since $\Grop_\Psi$ is closable, we will always consider its closed extension and denote it with the same symbol.

\

 As we have seen, without any assumption on $\Psi$ the synthesis operator $\Sy_\Psi$ is densely defined in $\ell_2(\ind).$ It is well known (see e.g. \cite{Christensen03}) that for $\Sy_\Psi$ to be a well defined operator from $\ell_2(\ind)$ to $\Hil$ one needs to assume  that
\begin{equation}\label{eq:hp}
\textnormal{the series} \ \displaystyle\sum_{i \in \ind} c_i \psi_i \ \textnormal{converges in} \ \Hil \ \textnormal{for all} \ c \in \ell_2(\ind) .
\end{equation}
In this case, by the uniform boundedness principle, $\Sy_\Psi$ is a bounded operator from $\ell_2(\ind)$ to $\Hil_\Psi$. Its adjoint operator $\An_\Psi : \Hil \to \ell_2(\ind)$ is then also bounded and this is equivalent to say that $\Psi$ satisfies the Bessel condition, that is, there exists a constant $B > 0$ such that
$$ 
\sum_{j \in \ind} |\langle \varphi, \psi_j \rangle_{\Hil} |^2 \leq B \|\varphi\|^2_{\Hil}, \quad \mbox{for all}\ \varphi \in \Hil.
$$
Since assuming $\Psi$ to a Bessel sequence implies that $\Sy_\Psi$ is bounded, one then has that condition (\ref{eq:hp}) is equivalent to the Bessel condition. 

Assuming (\ref{eq:hp}), one can then define the \emph{frame operator} as the bounded positive selfadjoint operator
$$
\begin{array}{rccc}
\Frame_\Psi = \Sy_\Psi \An_\Psi : & \Hil & \rightarrow & \Hil_\Psi\vspace{4pt}\\
& \varphi & \mapsto & \displaystyle\sum_{i \in \ind} \langle\varphi,\psi_i\rangle_\Hil \psi_i
\end{array}
$$
and under this hypothesis one also has that $\Grop_\Psi : \ell_2(\ind) \to \ell_2(\ind)$ is a bounded positive selfadjoint operator.

\section{Riesz and Frame sequences: The Gramian.}\label{GramianFrames}

\subsection{Riesz sequences: The Gramian}\label{sec:RieszGrammian}

When $\Psi= \{\psi_j\}_{j \in \ind} \subset \Hil$ is a Riesz sequence with Riesz bounds $A$ and $B$, the synthesis operator $\Sy_\Psi$ is bounded from $\ell_2(\ind)$ to $\Hil$ with norm not exceeding $\sqrt{B}$, and then so is $\An_\Psi:\Hil_\Psi\to \ell_2(\ind)$.

Since the central term in the definition of Riesz sequence (\ref{eq:Riesz}) reads
$$
\|\sum_{j \in \ind} c_j\psi_j\|^2_{\Hil} = \sum_{j,k \in \ind} c_j \ol{c_k} \langle \psi_j, \psi_k\rangle_{\Hil} = \langle  \Grop_\Psi c , c \rangle_{\ell_2(\ind)} 
$$
it follows that  $\Psi= \{\psi_j\}_{j \in \ind} \subset \Hil$ a Riesz sequence with Riesz bounds $A$ and $B$ if and only if 
\begin{equation}\label{eq:RieszOperatorcondition}
A \Id_{\ell_2(\ind)} \leq \Grop_\Psi \leq B \Id_{\ell_2(\ind)}\,.
\end{equation}
This result can be found in e.g. (\cite[\S 2.3 Lem. 2]{Meyer93} or \cite[Th. 3.6.6]{Christensen03}).

\ 

Recall that a bounded linear operator $T$ in a Hilbert space $\Hil$ is called \textit{positive}, and written $T \geq 0$, if $\langle Tx , x \rangle \geq 0$ for all $x\in \Hil.$ For $T$ and $S$ two bounded linear operators in a Hilbert space $\Hil$ the notation $T\geq S$ means $T - S \geq 0.$ It is well known (see e.g \cite[Th. 12.32] {Rudin73}) that a bounded linear operator $T$ in a Hilbert space $\Hil$ is positive if and only if $T$ is self adjoint and $\sigma(T; \Hil) \subset [0, \infty).$ Here $\sigma(T; \Hil)$ denotes the spectrum of $T$ in $\Hil$. Using this result, it follows from (\ref{eq:RieszOperatorcondition}) that $\Psi= \{\psi_j\}_{j \in \ind} \subset \Hil$ a Riesz sequence with Riesz bounds $A$ and $B$ if and only if
\begin{equation}\label{eq:RieszSpectrum}
\sigma(\Grop_\Psi; \ell_2(\ind)) \subset [A, B].
\end{equation}

\subsection{Frame sequences: The Frame Operator}\label{sec:Frame-1-1}

When $\Psi= \{\psi_j\}_{j \in \ind} \subset \Hil$ is a frame sequence  with frame bounds $A$ and $B$ the right hand side inequality in (\ref{eq:frames}), that is the Bessel condition, implies that the analysis operator $\An_\Psi$ is bounded from $\Hil_\Psi$ into $\ell_2(\ind)$ with norm not exceeding $\sqrt{B},$ and can be extended linearly to $\Hil$, with the same norm, by setting $\An_\Psi (\varphi)= 0$ for $\varphi \in (\Hil_\Psi)^\perp$. Therefore, the frame operator $\Frame_\Psi = \Sy_\Psi \An_\Psi$ is bounded from $\Hil$ into $\Hil_\Psi$
and its norm does not exceed $B$. Since the central term in the definition of frame sequence (\ref{eq:frames}) reads
$$
 \sum_{j \in \ind} |\langle \varphi, \psi_j \rangle_{\Hil} |^2 = \sum_{j \in \ind} \langle \varphi, \psi_j \rangle_{\Hil}\langle \psi_j , \varphi\rangle_{\Hil} = \langle \Frame_\Psi \varphi , \varphi\rangle_{\Hil} \,
$$
it follows that a Bessel sequence $\Psi= \{\psi_j\}_{j \in \ind} \subset \Hil$ is a frame sequence if and only if 
\begin{equation}\label{eq:frameOperatorcondition}
A \Proj_{\Hil_\Psi} \leq \Frame_\Psi \leq B \Proj_{\Hil_\Psi}
\end{equation}
where $\Proj_{\Hil_\Psi}$ is the orthogonal projection of $\Hil$ onto $\Hil_\Psi$.
This result can be found in e.g. \cite[\S 3.2]{Daubechies92}, \cite[\S 8.1]{HW96}.
As argued at the end of Section \ref{sec:RieszGrammian}, \eqref{eq:frameOperatorcondition} is equivalent to
\begin{equation}\label{eq:frameSpectrum}
\sigma(\Frame_\Psi; \Hil_\Psi) \subset [A, B].
\end{equation}

\subsection{Frame sequences: The Gramian}\label{sec:Frame-1-2}

In this subsection we derive a classical characterization of frame sequences  in terms of the Gramian as a consequence of the following basic result.
\begin{prop}\label{lem:duallemma}
Let $\Hil_1$ and $\Hil_2$ be separable Hilbert spaces, let $K : \Hil_1 \to \Hil_2$ be a bounded linear operator and denote with $K^* : \Hil_2 \to \Hil_1$ its adjoint. Let us call $G = |K|^2 = K^*K$, and $F = |K^*|^2 = KK^*$. Then for fixed $0 < A \leq B < +\infty$ the following are equivalent
\begin{itemize}
\item[i.] $A \Proj_{\ol{\Ran(K)}} \leq F \leq B  \Proj_{\ol{\Ran(K)}}$ \  in $\Hil_2.$ 
\item[ii.] $A G \leq G^2 \leq B G$ \ in $\Hil_1$.
\item[iii.] $ \sigma (G; \Hil_1) \subset \{0\}\cup [A , B]$
\item[iv.] $ A \Proj_{\ol{\Ran(K^*)}} \leq G \leq B \Proj_{\ol{\Ran(K^*)}}$ \ in $\Hil_1$.
\end{itemize}
\end{prop}

\begin{proof}
$i. \Leftrightarrow ii.$ Observe that $i.$ is equivalent to
$$
A \langle \varphi , \varphi \rangle_{\Hil_2} \leq \langle F\varphi , \varphi \rangle_{\Hil_2} \leq B \langle \varphi , \varphi \rangle_{\Hil_2} \quad \mbox{for all} \ \varphi \in \ol{\Ran(K)} \subset \Hil_2\,,
$$	
and that $ii.$ is equivalent to
$$
A \langle Gc , c \rangle_{\Hil_1} \leq \langle G^2 c , c \rangle_{\Hil_1} \leq B \langle Gc , c \rangle_{\Hil_1} \quad \mbox{for all} \ c \in \Hil_1\, .
$$
The desired equivalence follows from	 	
$$
\langle G^2c, c\rangle_{\Hil_1} = \langle K^*KK^*K c, c\rangle_{\Hil_1} = \langle FK c, Kc\rangle_{\Hil_2}\,,
$$ 
and  
$$
\langle Gc, c\rangle_{\Hil_1} = \langle K^*Kc , c \rangle_{\Hil_1} = \langle Kc , Kc \rangle_{\Hil_2}.
$$ 

$ii. \Rightarrow iii.$ The right hand side of $ii.$ means that $G^2 - A G \geq 0$ in $\Hil_1.$ By \cite[Th. 12.32]{Rudin73}, $\sigma(G^2 - A G ; \Hil_1) \subset [0,\infty).$ But $\sigma (G^2 - A G ; \Hil_1) = \sigma(G; \Hil_1)^2 - A \sigma(G; \Hil_1)$ (see e.g. \cite[\S 33, Th. 1]{Halmos51}). Since $G\geq0$, $\sigma (G; \Hil_1) \subset [0,\infty).$ Thus, if $\lambda \in \sigma (G; \Hil_1)$, $\lambda \geq 0$ and $\lambda (\lambda - A)= \lambda^2 - A\lambda \geq 0.$ This implies $\lambda = 0$ or $\lambda \geq A.$ Arguing similarly with the left hand side inequality of $(b)$ one obtains $\lambda = 0$ or $\lambda \leq B.$ This proves the result.

\vspace{.1cm}

$iii. \Rightarrow iv.$ It is enough to prove $\sigma(G;{\ol{\Ran(K^*)}} ) \subset [A , B]$, since this inclusion is equivalent to $A \Id_{\ol{\Ran(K^*)}} \leq G|_{\ol{\Ran(K^*)}} \leq B  \Id_{\ol{\Ran(K^*)}}$ (see, e.g. the argument  of the equivalence between (\ref{eq:RieszOperatorcondition}) and (\ref{eq:RieszSpectrum})), and this is also equivalent to $ A \Proj_{\ol{\Ran(K^*)}} \leq G \leq B \Proj_{\ol{\Ran(K^*)}}$ because $\Ker(G)=\Ker(K)$.

Since $\sigma(G;{\ol{\Ran(K^*)}} ) \subset \sigma (G; \Hil_1),$  condition $iii.$ implies that  \\ \mbox{$\sigma(G;{\ol{\Ran(K^*)}} )\subset \{0\}\cup [A , B]$}. If $\lambda=0$ belongs to $\sigma(G;{\ol{\Ran(K^*)}})$, since $G|_{\ol{\Ran(K^*)}}$ is self-adjoint and $\lambda = 0$ is an isolated point of its spectrum, $\lambda=0$ is an eigenvalue of $G|_{\ol{\Ran(K^*)}}$ (see e.g. \cite[Problem 6.28] {Kubrusly11}). This implies that $\Ker(G|_{\ol{\Ran(K^*)}}) \neq \{0\}$. On the other hand,
$$
\Ker(G|_{\ol{\Ran(K^*)}}) \subset \ol{\Ran(K^*)} \cap \Ker(K) = \Ker(K)^\perp \cap \Ker (K) = \{0\},
$$
and we have a contradiction. Thus, $0\notin \sigma(G;{\ol{\Ran(K^*)}})$ and $iv.$ follows. 
$iv.\Rightarrow ii.$. Since $G$ and $\Proj_{\ol{\Ran(K^*)}}$ are positive operators that commute, hypothesis $ii.$ implies,
$$
A \Proj_{\ol{\Ran(K^*)}} G \leq G^2 \leq B \Proj_{\ol{\Ran(K^*)}} G\,.
$$
But $ \Proj_{\ol{\Ran(K^*)}} G = G$ because $\ol{\Ran (G)} = \ol{\Ran (K^*)}.$
\end{proof}

A direct consequence of the above result is the following well known characterization of frame sequences.

\begin{cor} \label{cor:framegramian}
Let $\Psi = \{\psi_j\}_{j \in \ind} \subset \Hil$ be a collection of elements on a Hilbert space $\Hil$ and define $V_\Psi = \ol{\Ran(\Sy_\Psi^*)} = \Ker(\Sy_\Psi)^\bot$. For fixed $0 < A \leq B < \infty,$ the following statements are equivalent:
\begin{itemize}
	\item[i.] $\Psi$ is a frame sequence with frame bound  $A$ and $B$ .
	\item[ii.] The Gramian $\Grop_\Psi$ is well defined on $\ell_2(\ind)$ and $\sigma( \Grop_\Psi ; \ell_2(\ind)) \subset \{0\}\cup [A , B].$
	\item[iii.] The Gramian $\Grop_\Psi$ is well defined on $\ell_2(\ind)$ and $A \Proj_{V_\Psi} \leq \Grop_\Psi \leq B \Proj_{V_\Psi}$. 
\end{itemize}
\end{cor}

\begin{proof}
	With $K=T_\Psi$, statement $i.$ of Proposition \ref{lem:duallemma} is the frame sequence condition (\ref{eq:frameOperatorcondition}). The equivalences in the Corollary are, precisely, the equivalences of  $i., iii.$ and $iv.$  in Proposition \ref{lem:duallemma}.
\end{proof}

Another proof of the equivalence $i. \Leftrightarrow ii.$ can be found in \cite[Lem. 5.5.4]{Christensen03}.

\vspace{.1cm}

It is now easy to show  that Riesz sequences are frame sequences. Indeed, by (\ref{eq:RieszOperatorcondition}), $\Psi$ is a Riesz sequence if and only if  $\sigma( \Grop_\Psi; \ell_2(\ind)) \subset [A,B]$;  therefore $\sigma( \Grop_\Psi; \ell_2(\ind)) \subset [A,B] \cup \{0\}$ and $ii.  \Rightarrow i.$ of Corollary \ref{cor:framegramian} gives that $\Psi$ is a frame sequence.

\section{Bracket map, the Gramian and cyclic systems} \label{Braket=Gramian}

In this section we will prove the main result of the paper. It says that when $(\Gamma,\Pi,\Hil)$ is a dual integrable triple, the operator Bracket map coincides with the Gramian  on a dense set of $\Hil$. Once this is established, we show how it allows to easily deduce the known characterizations of frame and Riesz bases on cyclic systems in terms of the operator Bracket map.

\subsection{Bracket map and the Gramian}\label{sec:G=B}
Consider a family $\Psi$ that is an orbit $\orb$ of a single vector $\psi \in \Hil$ under a unitary representation $\Pi$ of a discrete countable group $\Gamma$, i.e. $\orb = \{\Pi(\gamma)\psi\}_{\gamma \in \Gamma}$, and denote its linearly generated space by $\psis{\psi} = \ol{\vsp\,\orb}^\Hil$.

The associated Gram matrix reads
$$
\Gram_\orb^{\gamma' , \gamma} = \langle \Pi(\gamma)\psi, \Pi(\gamma')\psi\rangle_\Hil = \g(\gamma^{-1}\gamma') \ , \quad \gamma,\gamma' \in \Gamma
$$
where we have introduced the notation
$$
\g(\gamma) = \langle \psi, \Pi(\gamma)\psi\rangle_\Hil .
$$
The function $\g \in \ell_\infty(\Gamma)$ is the prototype of a function of positive type (see e.g. \cite[\S 3.3]{Folland95}), and $\Gram_\orb$ is the associated positive definite kernel.

If we assume that $\orb$ satisfies condition (\ref{eq:generalsquareintegrability}), the Gramian operator $\Grop_\orb$ is then a densely defined right convolution operator on $\ell_2(\Gamma)$. Indeed, if $f = \{f(\gamma)\}_{\gamma \in \Gamma} \in \ell_0(\Gamma)$
$$
\Grop_\orb f (\gamma') = \sum_{\gamma \in \Gamma} f(\gamma) \langle \psi, \Pi(\gamma^{-1}\gamma')\psi\rangle_\Hil
= f \ast \g (\gamma') .
$$

We can then prove the main theorem of this section.

\begin{theo}\label{prop:G=B}
Let $(\Gamma,\Pi,\Hil)$ be a dual integrable triple.
\begin{itemize}
\item[i.] If $\psi$ is such that $\Grop_\orb$ is a closed and densely defined operator on $\ell_2(\ind)$, then
$$
[\psi,\psi] = \Grop_\orb .
$$
\item[ii.] If $\psi$ is such that $[\psi,\psi] \in L^2(\vn(\Gamma))$, then $\Grop_\orb$ is a closed and densely defined operator on $\ell_2(\ind)$.
\end{itemize}
\end{theo}
\begin{proof}
To prove $i.$, let us first see that $\Grop_\orb \in L^1(\vn(\Gamma))$. As a right convolution operator, it is affiliated with $\vn(\Gamma)$. So, since it is a positive operator, it suffices to check that its trace is finite. This is true because
$$
\tau(\Grop_\orb) = \langle \An_\orb \Sy_\orb \delta_\id, \delta_\id \rangle_{\ell_2(\Gamma)} = \|\Sy_\orb \delta_\id\|^2_{\Hil} = \|\psi\|_\Hil^2 .
$$
In order to see the desired claim, by $L^1(\vn(\Gamma))$ uniqueness of Fourier coefficients (see e.g. \cite[Lem. 2.1]{BHP14}) we need only to prove that
$$
\langle \psi, \Pi(\gamma)\psi\rangle_\Hil = \tau(\Grop_\orb\rr(\gamma)) \quad \forall \, \psi \in \Hil \, , \ \forall \, \gamma \in \Gamma .
$$
Since $\Grop_\orb \delta_\id (\gamma) = \langle \psi, \Pi(\gamma)\psi\rangle_\Hil$, using the traciality of $\tau$ we have indeed
\begin{displaymath}
\tau(\Grop_\orb\rr(\gamma)) = \tau(\rr(\gamma)\Grop_\orb) = \langle \Grop_\orb \delta_\id, \delta_\gamma\rangle_{\ell_2(\Gamma)} = \langle \psi, \Pi(\gamma)\psi\rangle_\Hil .
\end{displaymath}
To prove $ii.$ observe that, by Plancherel Theorem (see e.g. \cite[Lem. 2.2]{BHP14}), $[\psi,\psi] \in L^2(\vn(\Gamma))$ implies that $\sum_{\gamma \in \Gamma}|\tau([\psi,\psi]\rr(\gamma))|^2 < \infty$. By definition of dual integrability, this is equivalent to $\sum_{\gamma \in \Gamma}|\langle \psi, \Pi(\gamma)\psi\rangle_\Hil|^2 < \infty$, which coincides with condition (\ref{eq:generalsquareintegrability}), so the conclusion follows by Corollary \ref{cor:closableGram}.
\end{proof}

\begin{rem}
 Note that the above theorem says exactly that the Bracket map and the Gramian agree on a dense set of $\Hil$. Indeed,  by \cite[Th. 4.1]{BHP14} and Plancherel theorem the set of $\psi \in \Hil$ such that $[\psi,\psi] \in L^2(\vn(\Gamma))$ is a dense set in $\Hil$.
\end{rem}

\subsection{Characterizations of Riesz and frame cyclic systems}
We  show here how to obtain the characterization results of Riesz and frame cyclic systems in terms of the operator Bracket map, by combining the results of Sections \ref{GramianFrames} and \ref{sec:G=B}. This will provide simpler proofs for known results that were previously proven using sophisticated techniques.

\begin{prop}\label{cor:principalcharacterization}
Let $(\Gamma,\Pi,\Hil)$ be a dual integrable triple with associated Bracket map $[\cdot, \cdot]$. 
Then, the orbit $\orb$ is
\begin{itemize}
 \item[i.] a frame sequence with frame bounds $0<A\leq B$ if and only if
\begin{equation}\label{eq:opframes}
A \Proj_{(\Ker[\psi,\psi])^\bot} \leq [\psi,\psi] \leq B \Proj_{(\Ker[\psi,\psi])^\bot} .
\end{equation}
\item[ii.] a Riesz sequence with Riesz bounds $0<A\leq B$ if and only if
\begin{equation*}
A \Id_{\ell_2(\Gamma)}\leq [\psi,\psi] \leq B \Id_{\ell_2(\Gamma)}.
\end{equation*}
\end{itemize}
\end{prop}
\begin{proof}
To prove {\it i.} let us first assume that $\orb$ is a frame sequence. Then $\Grop_\orb$ is bounded, and 
by Corollary \ref{cor:framegramian} we have that $A \Proj_{(\Ker(\Grop_\orb))^\bot} \leq \Grop_\orb \leq B \Proj_{(\Ker(\Grop_\orb))^\bot}$, where we have used that $\Ker(\Grop_\orb)=\Ker (T_\orb)$. Furthermore, since in particular $\Grop_\orb$  belongs to $\vn(\Gamma) \subset L^2(\vn(\Gamma))$, by Theorem \ref{prop:G=B} we know that $[\psi ,\psi]=\Grop_\orb$, and then \eqref{eq:opframes} follows.

Reciprocally, if \eqref{eq:opframes} holds, then $[\psi,\psi]\in \vn(\Gamma) \subset L^2(\vn(\Gamma))$ and by Theorem \ref{prop:G=B} $[\psi ,\psi]=\Grop_\orb$. Thus, Corollary \ref{cor:framegramian} implies that $\orb$ is a frame sequence.
 To see {\it ii.} simply observe that Riesz sequences is equivalent to \eqref{eq:RieszOperatorcondition} and that  $[\psi ,\psi]=\Grop_\orb$ by Theorem \ref{prop:G=B}.
\end{proof}

\begin{rem}
 Observe that the above result recovers those  of \cite[Th. A]{BHP14}. 
\end{rem}

\section{Abelian groups}\label{abelian}

In this section we show the relationship between the abelian Bracket map and the operator Bracket map. More precisely, we show that for abelian groups the Bracket map (\ref{eq:Bracketabelian}) is the Fourier multiplier of the operator Bracket of Definition \ref{noncommutativeBracket}. As a consequence, it is possible to obtain an explicit proof that in abelian settings the condition on the Brackets in Proposition \ref{cor:principalcharacterization} is equivalent to that of Proposition \ref{prop:conmutative}.

Let $\Gamma$ be a discrete and countable abelian group, let us denote with $\wh{\Gamma}$ its dual group of characters, and let $\F_\Gamma : \ell_2(\Gamma) \to L^2(\wh{\Gamma})$ be the Fourier transform
$$
\F_\Gamma u = \sum_{\gamma \in \Gamma} u(\gamma)\overline{X_\gamma} \ , \quad u \in \ell_2(\Gamma),
$$
where $X_\gamma:\wh \Gamma\to \C$ are the continuous characteres of $\wh\Gamma$, that are $X_\gamma(\alpha)=\alpha(\gamma)$ for $\alpha\in \wh\Gamma$.
Let $S(\Gamma) = \vsp\{ \rr(\gamma)\}_{\gamma \in \Gamma} \subset \vn(\Gamma)$ and let $P(\widehat \Gamma) = \vsp\{X_\gamma\}_{\gamma \in \Gamma} \subset L^\infty(\wh{\Gamma})$ denote respectively the sets of noncommutative and commutative trigonometric polynomials. Define then the map $\Lambda : S(\Gamma) \to P(\wh{\Gamma})$ as
\begin{equation}\label{eq:multiplier}
\Lambda : F = \sum_{\gamma \in \Delta_F} \widehat F(\gamma) \rr(\gamma)^* \ \mapsto \ 
\Lambda (F) = \sum_{\gamma \in \Delta_F} \widehat F(\gamma)\, \overline{X_\gamma}
\end{equation}
where  $\Delta_F \subset \Gamma$ is a finite set. This map is the usual multiplier map, that turns a convolution operator into the multiplier by the Fourier transform of the convolution kernel, i.e. it satisfies
\begin{equation}\label{eq:Fourierintertwining}
\F_\Gamma F u  = \Lambda(F)\F_\Gamma u, \quad F \in S(\Gamma), \quad u \in  \ell_2(\Gamma).
\end{equation}
As such, it extends by density to a bounded map from $\vn(\Gamma)$ to $L^\infty(\wh{\Gamma})$, also denoted by $\Lambda$, which is an isometry and obviously  preserves equation \eqref{eq:Fourierintertwining} on $\vn(\Gamma)$. Indeed, if $F \in \vn(\Gamma)$ then it is a bounded operator on $\ell_2(\Gamma)$ and, by definition, $\|F\|_{L^\infty(\vn(\Gamma))} = \|F\|_{op}$, where $\|\cdot\|_{op}$ denotes the operator norm. So,  we have
$$
\|F\|_{L^\infty(\vn(\Gamma))} = \|\F_\Gamma F \F_\Gamma^{-1}\|_{op} = \textnormal{ess}\sup_{\alpha \in \wh{\Gamma}} |\Lambda(F)(\alpha)| = \|\Lambda(F)\|_{L^\infty(\wh{\Gamma})},
$$
where we have used the known fact that the operator norm of a multiplier operator is the $L^{\infty}$-norm of the multiplication kernel, which in our case is $\Lambda(F)$.

For the sake of completeness we recall the following well-known result.
\begin{prop}
The map $\Lambda : \vn(\Gamma) \to L^\infty(\wh{\Gamma})$ defined by density as in (\ref{eq:multiplier}) is an isometric $*$-homomorphism satisfying
\begin{equation}\label{eq:LambdaFouriercoefficients}
\int_{\widehat \Gamma} \Lambda (F)(\alpha) \alpha(\gamma) d\alpha = \wh{F}(\gamma) \ , \quad \forall \ \gamma \in \Gamma .
\end{equation}
\end{prop}
\begin{proof}
It is easy to see that $\Lambda$ is a linear bijective map, and we have already seen that $\|F\|_{L^\infty(\vn(\Gamma))} = \|\Lambda(F)\|_{L^\infty(\wh{\Gamma})}$. Moreover, it follows immediately by the definition that $\Lambda$ is an algebra homomorphism, i.e.
$$
\Lambda (F_1 F_2) = \Lambda (F_1) \Lambda (F_2) \ , \quad F_1, F_2 \in \vn(\Gamma) ,
$$
and that $\Lambda$ preserves the natural involutions given by operator adjoint in $\vn(\Gamma)$ and complex conjugation in $L^\infty(\wh{\Gamma})$, i.e.
$$
\Lambda (F^*) = \overline{\Lambda (F)}.
$$
Finally, identities (\ref{eq:LambdaFouriercoefficients}) can be obtained by (\ref{eq:Fourierintertwining}) since, for all $u \in \ell_2(\Gamma)$
$$
(\wh{F}\ast u)(\gamma) = Fu(\gamma) = \F_\Gamma^{-1}(\Lambda(F)\F_\Gamma u)(\gamma) = (\F_\Gamma^{-1}\Lambda(F)\ast u)(\gamma)
$$
so (\ref{eq:LambdaFouriercoefficients}) results by choosing $u = \delta_\id$.
\end{proof}

This result can actually be extended to all noncommutative $L^p$ spaces over $\vn(\Gamma)$, as follows.
\begin{theo}\label{theo:Lpmult}
Let $1 \leq p \leq \infty$. The multiplier map $\Lambda$ extends to an isometric isomorphism from $L^p(\vn(\Gamma))$ onto $L^p(\wh{\Gamma})$ satisfying (\ref{eq:LambdaFouriercoefficients}) and such that
\begin{equation}\label{eq:Holderhomomorphism}
\Lambda(F_1 F_2)  = \Lambda (F_1)\Lambda (F_2),
\end{equation}
for all $F_1\in L^p(\vn(\Gamma))$, $F_2\in L^q(\vn(\Gamma))$, with $1 \leq p , q \leq \infty$  such that $\frac1p + \frac1q = 1$.
\end{theo}
\begin{proof}
Let us first prove that $\Lambda$ maps $L^1(\vn(\Gamma))$ isometrically into $L^1(\wh{\Gamma})$, i.e.
$$
\|F\|_{L^1(\vn(\Gamma))} = \|\Lambda(F)\|_{L^1(\wh{\Gamma})} .
$$
It suffices to prove it in $S(\Gamma)$, and extend the result by density, which also immediately provides surjectivity. For $F\in S(\Gamma)$, let $R_F = \|\Lambda(F)\|_{L^\infty(\widehat \Gamma)} < \infty$. Since $|z|$ is a continuous function in $\C$, by Weierstrass Approximation Theorem, given $\epsilon > 0$ there exist a polynomial $\displaystyle Q_\epsilon (z) = \sum_{k=1}^N b_k z^k$ such that $||z| - Q_\epsilon (z)| \leq \epsilon$ on the ball of radius $R_F$. Using that $|\Lambda (F) (\alpha)| \leq R_F$ for all $\alpha \in \widehat \Gamma$ and that $d\alpha$ is the normalized Haar measure on the compact measure space $\widehat \Gamma$ we deduce:
$$
\Big| \int_{\widehat \Gamma} |\Lambda(F)(\alpha)| d\alpha - \int_{\widehat \Gamma} Q_\epsilon (\Lambda (F)(\alpha)) d\alpha \Big| \leq \epsilon \int_{\widehat \Gamma} d\alpha = \epsilon\,.
$$
By (\ref{eq:LambdaFouriercoefficients}) with $\gamma = \id$ and the homomorphism property of $\Lambda$ we obtain
$$ \int_{\widehat \Gamma} Q_\epsilon (\Lambda (F)(\alpha)) d\alpha = \sum_{k=1}^N b_k \tau(F^k) = \tau (Q_\epsilon (F))\,.$$
Since the operator $Q_\epsilon (F) - |F|$ is an operator with spectral radius less than or equal $\epsilon$ we obtain
$$ | \tau (Q_\epsilon (F)) - \tau(|F|) | \leq \tau (|Q_\epsilon (F) - |F||) \leq \epsilon \tau(I) = \epsilon\,.$$
Then, since $\epsilon$ is arbitrary we deduce,
$$ \| \Lambda (F) \|_{L^1 (\widehat \Gamma)} = \tau(|F|) = \| F\|_1\,.$$
Now, using that  we already proved the isometry property of $\Lambda$  for $p = \infty$, the result for all $1 < p < \infty$ follows  by interpolation.

What is left is  to prove (\ref{eq:Holderhomomorphism}). For this, observe first that $F_1 F_2 \in L^1(\vn(\Gamma))$ by H\"{o}lder's inequality. Thus, it is enough to show that $\Lambda(F_1 F_2)$ and $\Lambda (F_1) \Lambda (F_2)$ coincide as functions in $L^1(\widehat \Gamma).$ Given $\epsilon >0$, choose $P_1, P_2 \in S(\Gamma)$ such that
$$
\| F_1 - P_1 \|_p \leq \epsilon, \qquad \mbox{and} \qquad \| F_2 - P_2 \|_q \leq \epsilon\,.
$$
By the linearity of $\Lambda$ and the homomorphism property,
$$
\Lambda (F_1 F_2) = \Lambda ((F_1 - P_1) F_2) + \Lambda (P_1(F_2 - P_2)) + \Lambda(P_1) \Lambda(P_2).
$$
We also have,
$$
\Lambda(F_1) \Lambda (F_2) = (\Lambda(F_1) - \Lambda(P_1)) \Lambda (F_2) + \Lambda (P_1) (\Lambda(F_2) - \Lambda(P_2)) + \Lambda(P_1) \Lambda(P_2).
$$
Therefore,
\begin{align} \label{lem:LCA2-2}
\| \Lambda (F_1 F_2) &-  \Lambda(F_1) \Lambda (F_2)\|_{L^1(\widehat \Gamma)} \leq
\| \Lambda ((F_1 - P_1) F_2)\|_{L^1(\widehat \Gamma)} + \|\Lambda (P_1(F_2 - P_2))\|_{L^1(\widehat \Gamma)}  \nonumber \\
& +\|(\Lambda(F_1) - \Lambda(P_1)) \Lambda (F_2)\|_{L^1(\widehat \Gamma)} +
\|  \Lambda (P_1) (\Lambda(F_2) - \Lambda(P_2))   \|_{L^1(\widehat \Gamma)}.
\end{align}
To bound the right hand side of \eqref{lem:LCA2-2} apply the isometry property of $\Lambda$ and 
H\"{o}lder's inequality to obtain 
$$
\|\Lambda ((F_1 - P_1) F_2)\|_{L^1(\widehat \Gamma)} \leq \epsilon \|F_2\|_q,
$$
$$
\|\Lambda (P_1(F_2 - P_2))\|_{L^1(\widehat \Gamma)} \leq \epsilon \|P_1\|_p \leq \epsilon (\|F_1\|_p + \epsilon),
$$
$$
\|(\Lambda(F_1) - \Lambda(P_1)) \Lambda (F_2)\|_{L^1(\widehat \Gamma)} \leq \epsilon \|F_2\|_q,
$$
and
$$
\|  \Lambda (P_1) (\Lambda(F_2) - \Lambda(P_2))   \|_{L^1(\widehat \Gamma)} \leq \epsilon \|P_1\|_p \leq \epsilon (\|F_1\|_p + \epsilon). $$
The result follows from these inequalities since $\epsilon$ is arbitrary.
\end{proof}

Recall that by \cite[Corollary 3.4]{HSWW10} and \cite[Theorem 4.1]{BHP14} we have that, for abelian groups, the two notions of dual integrability coincide since they are both equivalent to the square integrability of the representation. However, for a discrete abelian group we have two definitions of a Bracket map of different nature: the one given in \cite{BHP14} (see Definition \ref{noncommutativeBracket}), which is operator valued, that we will denote by $[ \ , \ ]^{op}$ in this subsection, and the one defined in \cite{HSWW10} (see (\ref{eq:Bracketabelian})), whose values are functions, and denoted by $[ \ , \ ]$ in this subsection. Then, we now prove the main result of this section which establishes the relationship between  $[ \ , \ ]^{op}$ and $[ \ , \ ]$.

\begin{theo}
Let $(\Gamma, \Pi, \Hil)$ be a dual integrable triple and suppose that $\Gamma$ is abelian.
If $\psi_1, \psi_2 \in \Hil$, then
\begin{equation} \label{Brackets}
\Lambda([\psi_1, \psi_2]^{op}) = [\psi_1, \psi_2]
\end{equation}
as elements of $L^1(\wh{\Gamma}).$
\end{theo}
\begin{proof}
Observe first that, by definition of operator Bracket map, we have
$$
\wh{[\psi_1,\psi_2]^{op}}(\gamma) = \langle \psi_1, \Pi(\gamma)\psi_2\rangle_\Hil.
$$
Now, by Theorem \ref{theo:Lpmult}, we can apply the identity (\ref{eq:LambdaFouriercoefficients}) to obtain
$$
\int_{\widehat \Gamma} \Lambda([\psi_1, \psi_2]^{op})(\alpha) \alpha(\gamma)\, d\alpha = \wh{[\psi_1,\psi_2]^{op}}(\gamma) = \langle \psi_1, \Pi(\gamma)\psi_2\rangle_\Hil .
$$
Again by Theorem \ref{theo:Lpmult}, $\Lambda([\psi_1, \psi_2]^{op}) \in L^1(\wh{\Gamma})$. Then, the desired claim follows by uniqueness of Fourier coefficients (see \eqref{eq:Bracketabelian}).
\end{proof}

Finally, we observe that the equivalence of the conditions obtained in \cite{HSWW10} and the ones obtained in \cite{BHP14} for a dual integrable unitary orbit of a discrete abelian group to form an ortonormal, Riesz or frame systems, in terms of the Bracket maps can be checked directly in terms of the multiplier map $\Lambda$ as follows.
\begin{prop} \label{prop:LCA5}
Let $(\Gamma, \Pi, \Hil)$ be a dual integrable triple and suppose that $\Gamma$ is abelian. Given $\psi \in \Hil, \,\,\psi\neq0$, let $\chi_{\Omega_\psi} = \{ \alpha \in \widehat \Gamma : [\psi, \psi](\alpha) > 0 \}.$ For $0 < A \leq B < \infty,$ the following are equivalent:
\begin{itemize}
\item[i.] $A s_{[\psi, \psi]^{op}} \leq [\psi , \psi]^{op} \leq B s_{[\psi, \psi]^{op}}$
\item[ii.] $A \chi_{\Omega_\psi} (\alpha) \leq [\psi , \psi] (\alpha) \leq B \chi_{\Omega_\psi}(\alpha)\,, \mbox{a.e.} \  \alpha \in \widehat \Gamma\,.$
\end{itemize}
Here $s_{[\psi, \psi]^{op}}$ is the support of the operator $[\psi, \psi]^{op}$, as defined in \eqref{eq:support}.
\end{prop}

In order to prove Proposition \ref{prop:LCA5}, observe first that, as a consequence of (\ref{eq:Fourierintertwining}) and of Plancherel Theorem, for all $u \in \ell_2(\Gamma)$ and all $F \in \vn(\Gamma)$ we have
\begin{align*}
\langle F u, u\rangle_{\ell_2(\Gamma)} & = \langle \F_\Gamma F u, \F_\Gamma u\rangle_{L^2(\wh{\Gamma})} = \langle \Lambda(F) \F_\Gamma u, \F_\Gamma u\rangle_{L^2(\wh{\Gamma})}\\
& = \int_{\wh{\Gamma}} \Lambda(F)(\alpha) |\F_\Gamma u(\alpha)|^2 d\alpha .
\end{align*}
This provides a simple argument for the proof of the following useful lemma, which could be seen actually as a consequence of the Spectral Theorem.
\begin{lem}\label{lem:inequality}
Let $\Gamma$ be a discrete and countable abelian group.
If $F \in \vn(\Gamma)$ is a selfadjoint operator on $\ell_2(\Gamma)$, then
$$
F \geq 0 \quad \iff \quad \Lambda(F)(\alpha) \geq 0 \quad \textnormal{a.e.} \ \alpha \in \wh{\Gamma}.
$$
\end{lem}

A second general fact that is needed is that $\Lambda$ maps the support of a self-adjoint operator in $L^1(\vn(\Gamma))$ to a characteristic function.
\begin{lem} \label{lem:LCA4}
Let $\Gamma$ be a discrete and countable abelian group and suppose that $F$ is a self-adjoint positive operator in $L^1(\vn(\Gamma)).$ Let $\Omega_F = \{\alpha \in \widehat \Gamma : \Lambda(F)(\alpha) >0\}.$ Then
$$
\Lambda (s_F) (\alpha) = \chi_{\Omega_F} (\alpha)\,, \qquad \mbox{a.e.} \  \alpha \in \widehat \Gamma.
$$
\end{lem}
\begin{proof}
Since $\chi_{\Omega_F} \in L^\infty (\widehat \Gamma)$ we can choose $q \in \vn(\Gamma)$ such that $\Lambda (q) = \chi_{\Omega_F}.$ Using the properties of $\Lambda$, it is easy to show that $q^2 = q$ and $q^* = q$; thus $q$ is a projection. Moreover $ \Lambda (F q) = \Lambda (F) \Lambda (q) = \Lambda (F) \chi_{\Omega_F} = \Lambda (F)$, as functions in $L^1(\widehat \Gamma)$. Since $s_F$ is the support of $F$, we deduce $s_F \leq q.$ Hence, $q - s_F$ is a positive self-adjoint operator in $\vn(\Gamma)$, and by Lemma \ref{lem:inequality}, $\Lambda (s_F) (\alpha) \leq \Lambda (q)(\alpha) = \chi_{\Omega_F}(\alpha)$ a.e. $\alpha \in \widehat \Gamma$.

On the other hand,  by definition of support, we have $\Lambda (F) = \Lambda (F s_F) = \Lambda (F) \Lambda (s_F)$, so that 
$\Lambda (F)(\alpha) [ \Lambda (s_F)(\alpha) -1] = 0$ a.e. $\alpha \in \widehat \Gamma.$ Thus, $\Lambda (s_F)(\alpha) -1 = 0$ when $\Lambda (F)(\alpha) >0$.This implies $\Lambda (s_F)(\alpha) \geq \chi_{\Omega_F}(\alpha)$ a.e. $\alpha \in \widehat \Gamma$ since, by Lemma \ref{lem:inequality}, we know that $\Lambda(s_F)\geq0$ a.e. $\alpha \in \widehat \Gamma$.
\end{proof}

\begin{proof}[Proof of Proposition \ref{prop:LCA5}]
Assuming $i.$, the operators $[\psi , \psi]^{op} - A s_{[\psi, \psi]^{op}}$ and $B s_{[\psi, \psi]^{op}} - [\psi , \psi]^{op}$ are self-adjoint and positive elements of $\vn(\Gamma)$. By Lemma \ref{lem:inequality} we then have
$$
A\, \Lambda (s_{[\psi, \psi]^{op}})(\alpha) \leq \Lambda ([\psi , \psi]^{op}) (\alpha) \leq B\,\Lambda (s_{[\psi, \psi]^{op}})(\alpha)\,, \quad \mbox{a.e.} \ \alpha \in \widehat \Gamma\,.
$$
By (\ref{Brackets}) and Lemma \ref{lem:LCA4}, this is equivalent to
$$
A\, \chi_{\Omega_{[\psi, \psi]^{op}}}(\alpha) \leq [\psi , \psi] (\alpha) \leq B\,\chi_{\Omega_{[\psi, \psi]^{op}}}(\alpha)\,, \quad \mbox{a.e.} \ \alpha \in \widehat \Gamma\,.
$$
But, by Lemma \ref{lem:LCA4} again, we have $\Omega_{[\psi, \psi]^{op}} = \Omega_{\psi}$, which proves $ii.$.
An easy adaptation of this same argument allows to prove that $ii. \Rightarrow i.$.
\end{proof}

\

\paragraph{Acknowledgements.} D. Barbieri was supported by a Marie Curie Intra Euro\-pean Fellowship (N. 626055) within the 7th European Community Framework Programme. D. Barbieri and E. Hern\'andez were supported by Grant MTM2013-40945-P (Ministerio de Econom\'ia y Competitividad, Spain). V. Paternostro by Grants UBACyT  2002013010022BA and CONICET-PIP 11220110101018.

\newpage

\noindent
\textbf{Davide Barbieri}\\
Universidad Aut\'onoma de Madrid, 28049 Madrid, Spain\\
\href{mailto:davide.barbieri@uam.es}{\tt davide.barbieri@uam.es}

\

\noindent
\textbf{Eugenio Hern\'andez}\\
Universidad Aut\'onoma de Madrid, 28049 Madrid, Spain\\
\href{mailto:eugenio.hernandez@uam.es}{\tt eugenio.hernandez@uam.es}

\

\noindent
\textbf{Victoria Paternostro}\\
Universidad de Buenos Aires and 
IMAS-CONICET, Consejo  Nacional de Investigaciones Cient\'ificas y T\'ecnicas, 1428 Buenos Aires,  Argentina\\
\href{mailto:vpater@dm.uba.ar}{\tt vpater@dm.uba.ar}

\end{document}